\documentclass{article}

\usepackage{PRIMEarxiv}
\usepackage[utf8]{inputenc} 
\usepackage[T1]{fontenc}    
\usepackage{hyperref}       
\usepackage{url}            
\usepackage{booktabs}       
\usepackage{amsfonts}       
\usepackage{nicefrac}       
\usepackage{microtype}      
\usepackage{lipsum}
\usepackage{fancyhdr}       
\usepackage{graphicx}  

\graphicspath{{media/}}     
\usepackage{amsmath, amsthm}
\usepackage{algorithm}
\usepackage{algorithmic}
\usepackage{float}

\newtheorem{theorem}{Theorem}
\newtheorem{lemma}{Lemma}
\newtheorem{corollary}{Corollary}

\title{The Mean Value Theorem: Analytical Proof and Computational Approaches}

\author{
  Márcio Matheus de Lima Barboza \\
  Universidade Federal do Rio Grande do Norte \\
  Natal-RN\\
  \texttt{marcio.barboza.024@ufrn.edu.br} \\
   \And
  Francisco Marcio Barboza \\
   Universidade Federal do Rio Grande do Norte  \\
   Departamento de Computação e Tecnologia \\
  Caicó-RN\\
  \texttt{marcio.barboza@ufrn.br} \\
}

\begin{document}
\maketitle

\begin{abstract}
In this paper, we explore two fundamental theorems of differential calculus: Rolle's Theorem and the Mean Value Theorem (MVT). These theorems play a crucial role in the development of theoretical and practical results in mathematics, serving as the basis for various applications in analysis and modeling of real-world phenomena. Initially, we present the formal statements and their respective analytical proofs, highlighting the mathematical rigor necessary for understanding them. Additionally, we discuss the geometric interpretation of both theorems, emphasizing their importance in understanding properties of differentiable functions. The goal of this work is not only to validate these theorems through analytical methods but also to perform their computational verification, providing an integrated view between theory and practice.
\end{abstract}


\section{Introduction}

Rolle's Theorem and the Mean Value Theorem (MVT) are fundamental results in differential calculus, with profound theoretical and practical implications. Rolle's Theorem establishes conditions for the existence of at least one critical point in a differentiable function on a closed interval, while the Mean Value Theorem generalizes this result, connecting the average rate of change of a function with the slope of its derivative at a specific point \cite{spivak_calculus, apostol_analysis}.

The importance of these theorems transcends calculus, being used in various areas of mathematics, such as the analysis of functions, the study of differential equations, and optimization. Furthermore, they have direct applications in fields like physics and engineering, enabling the modeling of dynamic phenomena and the analysis of complex behaviors \cite{stewart_calculus, strang_calculus}.

In this paper, our objective is to explore Rolle's and the Mean Value Theorem from two complementary perspectives: the rigorous analytical proof and the geometric and computational verification. To this end, we will use mathematical and computational tools to illustrate the applicability and validity of these results in different contexts \cite{rolle_history}.

The proposed approach aims not only to reinforce theoretical understanding but also to highlight the connection between mathematical abstraction and its practical applications, emphasizing the pedagogical and scientific relevance of these theorems \cite{spivak_calculus, strang_calculus}.


\section{Lemmas and Theorems}
The lemmas and theorems presented in this section are taken from the book \cite{lages2011curso}.
\begin{lemma}
Let \( X \subseteq \mathbb{R} \), \( a \in \mathbb{R} \), \( f, g: X \to \mathbb{R} \).  
If 
\[
\lim_{x \to a} f(x) = L \quad \text{and} \quad \lim_{x \to a} g(x) = M,
\] 
with \( L < M \), then there exists \( \delta > 0 \) such that for all \( x \in X \),  
\[
0 < |x - a| < \delta \implies f(x) < g(x).
\]
\label{lemma1}
\end{lemma}

\begin{proof}
Let \( \varepsilon > 0 \) and \( c = \frac{L + M}{2} \), so that \( L < c < M \). Since \( \lim_{x \to a} f(x) = L \), there exists \( \delta_1 > 0 \) such that for all \( x \in X \),  
\[
0 < |x - a| < \delta_1 \implies f(x) \in (L - \varepsilon, L + \varepsilon).
\]
Similarly, since \( \lim_{x \to a} g(x) = M \), there exists \( \delta_2 > 0 \) such that for all \( x \in X \),  
\[
0 < |x - a| < \delta_2 \implies g(x) \in (M - \varepsilon, M + \varepsilon).
\]

By choosing \( \delta = \min(\delta_1, \delta_2) \), we have for all \( x \in X \) with \( 0 < |x - a| < \delta \),  
\[
f(x) \in (L - \varepsilon, L + \varepsilon) \quad \text{and} \quad g(x) \in (M - \varepsilon, M + \varepsilon).
\]
Since \( L + \varepsilon < M - \varepsilon \) (by the choice of \( c = \frac{L + M}{2} \)), it follows that  
\[
f(x) < g(x).
\]
The proof is complete.
\end{proof}

\begin{corollary}
Let \( X \subseteq \mathbb{R} \) and \( a \in \mathbb{R} \).  
If \( \lim_{x \to a} f(x) = L > 0 \), then there exists \( \delta > 0 \) such that for all \( x \in X \),  
\[
0 < |x - a| < \delta \implies f(x) > 0.
\]
\label{corollary1}
\end{corollary}

\begin{proof}
From Lemma, choose \( g(x) = 0 \) and \( M = 0 \) with \( L > 0 \). Thus, there exists \( \delta > 0 \) such that for all \( x \in X \) with \( 0 < |x - a| < \delta \),  
\[
f(x) > g(x) = 0.
\]
\end{proof}


\section{Rolle's Theorem and the Mean Value Theorem}

In this section, we present the statements of Rolle's Theorem and the Mean Value Theorem (MVT), followed by their respective analytical proofs based on \cite{bartle_elements, rudin_principles}. Both theorems are fundamental in differential calculus and provide the theoretical foundation for several subsequent results.

\begin{theorem}[Rolle]
Let \( f: [a, b] \to \mathbb{R} \) be a continuous function on \([a, b]\), such that \( f(a) = f(b) \).  
If \( f \) is differentiable on \( (a, b) \), then there exists a point \( c \in (a, b) \) where \( f'(c) = 0 \).
\label{rolle}
\end{theorem}

\begin{proof}
If \( f \) is constant on \([a, b]\), then \( f'(c) = 0 \) for all \( c \in (a, b) \), and the theorem is proven in this case.  

Otherwise, \( f \) is not constant, and by Fermat's Theorem, we know that at most \( f'(c) = 0 \) at the local maximum or minimum point of \( f \). Since \( f \) is continuous and differentiable, there exists a point \( c \in (a, b) \) such that \( f'(c) = 0 \).  
The proof is complete.
\end{proof}

\subsection{Relation Between the Theorems}

Rolle's Theorem can be viewed as a special case of the Mean Value Theorem, where \( f(a) = f(b) \). Both provide the foundation for more advanced analyses, such as proving the existence of roots and studying rates of change in physical phenomena.

\section{Pseudocode for Verifying the Mean Value Theorem}

The pseudocode \ref{pseucodigo} presented here aims to verify the Mean Value Theorem for a function \( f(x) \) defined on an interval \([a, b]\). The first step is to check if the function is continuous on \([a, b]\) and differentiable on \((a, b)\), which are the essential conditions for applying the theorem. Next, the average rate of change is calculated as
\[
m = \frac{f(b) - f(a)}{b - a},
\]
which represents the slope of the secant line passing through the points \( (a, f(a)) \) and \( (b, f(b)) \). The pseudocode then applies the bisection method to locate the point \( c \) within the interval \( (a, b) \), where the derivative \( f'(c) \) equals this average rate of change. The bisection method repeatedly divides the interval into two, evaluating the value of \( f'(c_{\text{mid}}) \) at the midpoint of the interval and adjusting the bounds \( c_{\text{left}} \) and \( c_{\text{right}} \) based on the value found, until the difference between the bounds is smaller than a pre-established precision \( \epsilon \). The final result is an approximation of the point \( c \) that satisfies the condition of the theorem, i.e., where \( f'(c) = m \). This algorithm provides an efficient solution for numerically finding the point \( c \) when an analytical solution is not possible.

\begin{algorithm}[H]
\caption{Verifying the Mean Value Theorem}
\begin{algorithmic}[1]
\REQUIRE Function \( f(x) \), interval \([a, b]\), precision \( \epsilon \)
\ENSURE Verifies if there exists \( c \in (a, b) \) such that \( f'(c) = \frac{f(b) - f(a)}{b - a} \)

\STATE Check if \( f(x) \) is continuous on \([a, b]\)
\IF{\( f(x) \) is not continuous on \([a, b]\)}
    \STATE Return "The Mean Value Theorem does not apply"
\ENDIF

\STATE Check if \( f(x) \) is differentiable on \((a, b)\)
\IF{\( f(x) \) is not differentiable on \((a, b)\)}
    \STATE Return "The Mean Value Theorem does not apply"
\ENDIF

\STATE Calculate the average rate of change:
\[
m = \frac{f(b) - f(a)}{b - a}
\]

\STATE Initialize the bounds of the interval:
\[
c_{\text{left}} = a, \quad c_{\text{right}} = b
\]

\WHILE{\( c_{\text{right}} - c_{\text{left}} > \epsilon \) \textbf{(desired precision)}}
    \STATE Calculate the midpoint:
    \[
    c_{\text{mid}} = \frac{c_{\text{left}} + c_{\text{right}}}{2}
    \]
    \IF{\( f'(c_{\text{mid}}) = m \)}
        \STATE Return \( c = c_{\text{mid}} \)
    \ELSIF{\( f'(c_{\text{mid}}) > m \)}
        \STATE Update \( c_{\text{right}} = c_{\text{mid}} \)
    \ELSE
        \STATE Update \( c_{\text{left}} = c_{\text{mid}} \)
    \ENDIF
\ENDWHILE

\STATE Return \( c = \frac{c_{\text{left}} + c_{\text{right}}}{2} \) as the approximation of the desired point
\end{algorithmic}
\label{pseucodigo}
\end{algorithm}

\section{Examples of Rolle's Theorem and the Mean Value Theorem}

In this section, we present three examples to illustrate the application of Rolle's Theorem and the Mean Value Theorem, using functions with different levels of complexity, along with computational verification.

\subsection{Example 1: Simple Function - Parabolic (Rolle's Theorem)}

Consider the function \( f(x) = x^2 - 4x + 3 \) defined on the interval \([1, 3]\). This function is continuous and differentiable on \([1, 3]\), and satisfies the condition \( f(1) = f(3) = 0 \). Rolle's Theorem guarantees that there exists a point \( c \in (1, 3) \) such that \( f'(c) = 0 \).

The derivative of the function is:
\[
f'(x) = 2x - 4.
\]
Solving \( f'(c) = 0 \), we find \( c = 2 \). Therefore, the point \( c = 2 \) satisfies the conditions of Rolle's Theorem.

\begin{figure}[H]
    \centering   
    \includegraphics[width=1\textwidth]{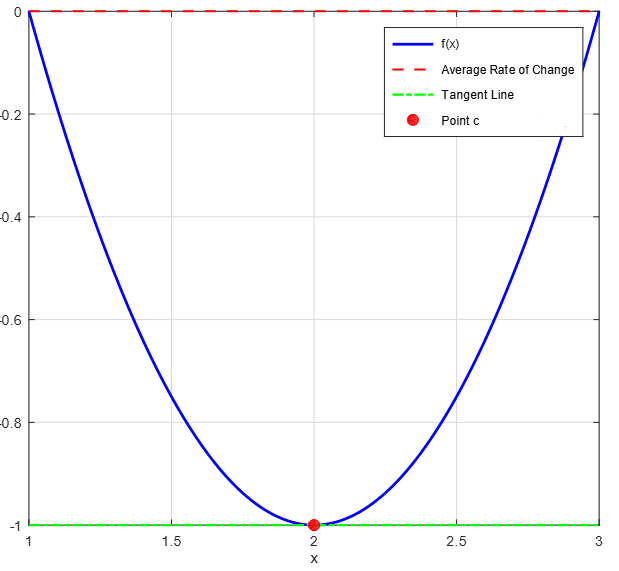}
    \caption{Execution of the algorithm in Example 1.}
    \label{fig:exemplo1}
\end{figure}

Figure~\ref{fig:exemplo1} illustrates the execution of the algorithm in Example 1.

\subsection{Example 2: Trigonometric Function (Mean Value Theorem)}

Consider the function \( f(x) = \sin(x) \) defined on the interval \([0, \pi/2]\). This function is continuous and differentiable on \([0, \pi/2]\). The slope of the secant line on the interval is given by:
\[
m = \frac{f\left(\frac{\pi}{2}\right) - f(0)}{\frac{\pi}{2} - 0} = \frac{\sin\left(\frac{\pi}{2}\right) - \sin(0)}{\frac{\pi}{2}} = \frac{1 - 0}{\frac{\pi}{2}} = \frac{2}{\pi}.
\]

By the Mean Value Theorem, there exists a point \( c \in (0, \pi/2) \) such that \( f'(c) = m \), i.e., \( f'(c) = \frac{2}{\pi} \).
The derivative of the function \( f(x) = \sin(x) \) is:
\begin{equation}
  f'(x) = \cos(x).  
\end{equation}
Now, we solve the equation \( \cos(c) = \frac{2}{\pi} \) to find the value of \( c \) in the interval \( (0, \pi/2) \). The solution to this equation gives the value of \( c \) that satisfies the conditions of the Mean Value Theorem.

\begin{figure}[H]
    \centering            \includegraphics[width=1\textwidth]{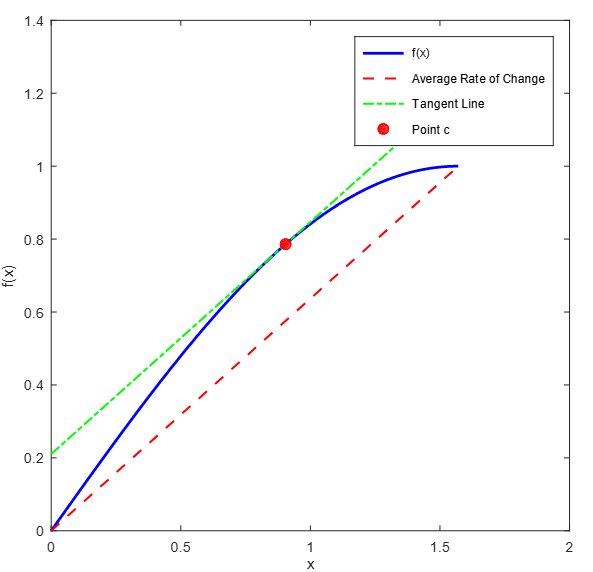}
    \caption{Execution of the algorithm in Example 2.}
    \label{fig:exemplo2}
\end{figure}

Figure~\ref{fig:exemplo2} illustrates the execution of the algorithm to find the point \( c \) in Example 2.

\section{Conclusion}

The examples presented show practical applications of Rolle's Theorem and the Mean Value Theorem, highlighting how the conditions of continuity, differentiability, and slope properties guarantee the existence of critical points or points that satisfy the average rate of change.

\bibliographystyle{unsrt}  
\bibliography{references}

\end{document}